\documentclass{amsart}
\usepackage{amsmath, amssymb, mathrsfs, verbatim, multirow}
\usepackage[all]{xy}
\usepackage{pifont}
\usepackage{float}
\usepackage{color}
\usepackage{enumitem}

\usepackage{tikz}
\usepackage{tikz-cd}
\usepackage{tkz-tab}
\usepackage{subcaption} 

\newtheorem{Teo}{Theorem}[section]
\newtheorem{Prop}[Teo]{Proposition}
\newtheorem{Lema}[Teo]{Lemma}
\newtheorem{Cor}[Teo]{Corollary}

\theoremstyle{definition}

\newtheorem{Obs}[Teo]{Remark}

\newcommand{\N}{\mathbb{N}}

\DeclareMathOperator{\CA}{char}

\begin{document}
\title{On stable and fixed polynomials}
\author{J. Novacoski}
\author{M. Spivakovsky}
\thanks{During the realization of this project the first author was supported by a grant from Funda\c c\~ao de Amparo \`a Pesquisa do Estado de S\~ao Paulo (process number 2017/17835-9).}

\begin{abstract}
Let $\nu$ be a rank one valuation on $K[x]$ and $\Psi_n$ the set of key polynomials for $\nu$ of degree $n\in\N$. We discuss the concepts of being $\Psi_n$-stable and $(\Psi_n,Q)$-fixed. We discuss when these two concepts coincide. We use this discussion to present a simple proof of Proposition 8.2 of \cite{MaSpiv} and Theorem 1.2 of \cite{michael}. 
\end{abstract}

\keywords{Key polynomials, stable polynomials, truncations of valuations, fixed polynomials}
\subjclass[2010]{Primary 13A18}

\maketitle
\section{Introduction}

Let $\nu$ be a rank one valuation on $K[x]$. For $n\in\N$ set
\[
\Psi_n=\{Q\in K[x]\mid Q\mbox{ is a key polynomial for }\nu\mbox{ and }\deg(Q)=n\}.
\] 
Suppose that $\Psi_n$ is non-empty and bounded (i.e., there exists $a\in K[x]$ such that $\nu(a)>\nu(Q)$ for every $Q\in \Psi_n$) and that $\nu(\Psi_n)$ does not have a maximum. Set $K[x]_n=\{a\in K[x]\mid \deg(a)<n\}$. For each $f\in K[x]$ and $Q\in \Psi_n$ the \textbf{$Q$-expansion} of $f$ is the expression
\[
f=a_0+a_1Q+\ldots+a_rQ^r=l(Q)
\]
where $l(X)\in K[x]_n[X]$. We denote the value $r$ (which does not depend on the choice of $Q\in\Psi_n$) in the previous expression by $\deg_X(f)$. The truncation of $\nu$ on $Q$ is given by
\[
\nu_Q(f):=\min_{0\leq i\leq r}\{\nu(a_iQ^i)\}.
\]
Let
\[
S_n=\{f\in K[x]\mid \nu_Q(f)<\nu(f)\mbox{ for every }Q\in\Psi_n\}.
\]
A polynomial $f$ is said to be \textbf{$\Psi_n$-stable} if it does not belong to $S_n$. A monic polynomial $F\in K[x]$ is called a \textbf{limit key polynomial} for $\Psi_n$ if it belongs to $S_n$ and has the smallest degree among polynomials in $S_n$.

In \cite{Kap}, Kaplansky introduces the concept of \textit{pseudo-convergent sequences}. These objects are strongly related to the set $\Psi_1$. For a given such sequence $\underline a$ we can define what it means for a polynomial  $f\in K[x]$ to be fixed by $\underline a$. Here, we generalize this concept for any of the sets $\Psi_n$. We say that $f=l(Q)$, $l(X)\in K[x]_n[X]$, is \textbf{$(\Psi_n,Q)$-fixed} if there exists $Q'\in \Psi_n$, $\nu(Q)<\nu(Q')$ such that $\nu(f)=\nu(l(Q'-Q))$.

Our first main result (Proposition \ref{Corsobrepowerofp}) is that we can choose a suitable $Q\in \Psi_n$ such that for any $f\in K[x]$ with $\deg_X(f)\leq \deg_X(F)$, we obtain that $f$ is $\Psi_n$-stable if and only if it is $(\Psi_n,Q)$-fixed.

We will fix a suitable $Q\in K[x]$ (see \eqref{eqaescolhaimpor} and \eqref{eqaescolhaimpor2}). Then we fix a limit ordinal $\lambda$ and any cofinal well-ordered (with respect to $\nu$) subset
\[
\{Q_\rho\}_{\rho<\lambda}\subseteq\{Q'\in \Psi_n\mid \epsilon(Q)<\epsilon(Q')\}.
\]
This means that if $\rho<\sigma<\lambda$, then $\nu(Q_\rho)<\nu(Q_\sigma)$ and that for every $Q'\in \Psi_n$, there exists $\rho<\lambda$ such that $\nu(Q')<\nu(Q_\rho)$. For each $\rho<\lambda$ set $h_\rho:=Q-Q_\rho\in K[x]_n$ and $\gamma_\rho=\nu(Q_\rho)$. It follows from the definition that $\{h_\rho\}_{\rho<\lambda}$ is a \textit{pseudo-convergent sequence} for $\nu$. For simplicity, we will denote $\nu_{Q_\rho}$ by $\nu_\rho$.
 
Let $p={\rm char}(K\nu)$ and $I$ be the set of all non-negative powers of $p$. As an application of Proposition \ref{Corsobrepowerofp} we can prove the following.
\begin{Teo}\label{Tehoremquasneunpas}
Let $F$ be a limit key polynomial for $\Psi_n$ and write $F=L(Q)$ for some $L(X)\in K[x]_n[X]$. Then we have the following.
\begin{description}
\item[(i)] There exists $\sigma<\lambda$ such that for every $\theta>\sigma$ the polynomial
\[
F_p=L(h_\theta)+\sum_{i\in I}\partial_iL(h_\theta)Q_\theta^i
\]
is a limit key polynomial for $\Psi_n$. Here $\partial_iL$ denotes the Hasse derivative of $L(X)$ (as a polynomial in $K(x)[X]$) of order $i$.

\item[(ii)] For each $i\in I\cup\{0\}$ there exists $a_i\in K[x]_n$  such that
\[
\overline F_p=\sum_{i\in I\cup\{0\}}a_iQ_\theta^i
\]
is a limit key polynomial for $\Psi_n$.
\end{description}
\end{Teo}

Kaplansky proved the above result in the case $n=1$. In that case, \textbf{(ii)} follows trivially from \textbf{(i)}. Our proof of Theorem \ref{Tehoremquasneunpas} follows Kaplansky's proof.

If $n>1$, then \textbf{(ii)} was proven in \cite{MaSpiv} (Proposition 8.2). An alternative proof of it was presented in \cite{michael} (Theorem 1.2). The advantage of our proof is that it is much simpler and presents as algorithm on how to construct the limit key polynomial of this form (from a given limit key polynomial). Also, our proof does not require that are in the equicharacteristic case.

\section{Preliminaries}
Throughout this paper $\nu$ will denote a rank one valuation on $K[x]$. For $f\in K[x]$ we denote
\begin{equation}\label{equancoepsilo}
\epsilon(f)=\max_{1\leq b\leq \deg(f)}\left\{\frac{\nu(f)-\nu(\partial_b f)}{b}\right\},
\end{equation}
where $\partial_bf$ denotes the Hasse derivative of $f$ of order $b$. We denote
\[
I(f)=\left\{b, 1\leq b\leq \deg(f)\mid \epsilon(f)=\frac{\nu(f)-\nu(\partial_b f)}{b}\right\}.
\]
A monic polynomial $Q\in K[x]$ is said to be a \textbf{key polynomial for $\nu$} if for every $f\in K[x]$, if $\epsilon(f)\geq \epsilon(Q)$, then $\deg(f)\geq\deg(Q)$. For a polynomial $f\in K[x]$ let
\[
f=f_0+f_1Q+\ldots+f_rQ^r
\]
be the $Q$-expansion of $f$. We set
\[
S_Q(f)=\{i,0\leq i\leq r\mid \nu_Q(f)=\nu(f_iQ^i)\}\mbox{ and }\delta_Q(f)=\max S_Q(f).
\]

Throughout this paper we will fix a limit key polynomial $F$ for $\Psi_n$ and denote $d:=\deg_X(F)$. Set
\[
B=\lim_{Q\in \Psi_n}\nu(Q)\mbox{ and }\overline B=\lim_{Q\in \Psi_n}\nu_Q(F).
\]
Take $Q_0\in \Psi_n$ and choose $Q\in \Psi_n$ such that
\begin{equation}\label{eqaescolhaimpor}
\epsilon(Q)-\epsilon(Q_0)>d(B-\nu(Q))
\end{equation}
and
\begin{equation}\label{eqaescolhaimpor2}
\epsilon(Q)-\epsilon(Q_0)>\overline B-\nu_Q(F).
\end{equation}
Write $F=L(Q)$ for $L(X)\in K[x]_n[X]$.

The next result is well-known. We will reprove it here because we need this slightly stronger statement.
\begin{Lema}\label{keypolmelhorado}
Let $Q$ be a key polynomial and take $f$ such that $f=qQ+r$ with $\gamma=\max\{\epsilon(f),\epsilon(r)\}<\epsilon(Q)$. Then we have
\[
\nu_Q(qQ)-(\epsilon(Q)-\gamma)\geq\nu(f)=\nu(r).
\]
\end{Lema}

\begin{proof}
Take $b\in I(qQ)$. Since $\epsilon(qQ)=\max\{\epsilon(Q),\epsilon(q)\}\geq \epsilon(Q)$ (Corollary 4.4 of \cite{Matheus}) we have
\[
\nu(qQ)-b\epsilon(Q)\geq \nu(\partial_b(qQ))\geq \min\{\nu(\partial_b(f)),\nu(\partial_b(r))\}\geq \min\{\nu(f),\nu(r)\}-b\gamma.
\]
Consequently, $\nu(f)=\nu(r)$ and
\[
\nu(qQ)-\nu(f)\geq b(\epsilon(Q)-\gamma)\geq \epsilon(Q)-\gamma.
\]
Applying the above discussion to $\nu_Q$ instead of $\nu$ we obtain the result.
\end{proof}
Take $f\in K[x]$ and $Q'\in \Psi_n$ with $\epsilon(Q)\leq \epsilon(Q')$ with $\epsilon(f)<\epsilon(Q_0)$ and write
\[
f=qQ'+r\mbox{ with }\deg(r)<\deg(Q')=\deg(Q_0).
\]
By \eqref{eqaescolhaimpor}, \eqref{eqaescolhaimpor2}, Lemma \ref{keypolmelhorado} and the fact that $\epsilon(r)<\epsilon(Q_0)$ we have
\begin{equation}\label{eqabespivmeseugost}
\nu_{Q'}(qQ')>\nu(f)+d(B-\nu(Q))
\end{equation}
and
\begin{equation}\label{eqabespivmeseugost2}
\nu_{Q'}(qQ')>\nu(f)+\overline B-\nu_Q(F).
\end{equation}
The next result is a well-known result about key polynomials.
\begin{Lema}\label{sugestnart}
Take $Q,Q'\in\Psi_n$ be such that $\nu(Q)<\nu(Q')$. For $f\in K[x]$ let $\displaystyle f=\sum_{i=0}^r f_iQ'^i$ be the $Q'$-expansion of $f$. Then
\[
\nu_Q\left(f\right)=\min_{0\leq i\leq r}\{\nu_Q(f_iQ'^i)\}.
\]
\end{Lema}
\begin{proof}
Since $Q'$ is monic and has the smallest degree among all polynomials $f$ such that $\nu_Q(f)<\nu(f)$, it is a (Mac Lane-Vaqui\'e) key polynomial for $\nu_Q$ (Theorem 31 of \cite{mabjulispiv}). In particular, $Q'$ is $\nu_{Q}$- minimal and the result follows from Proposition 2.3 of \cite{nart}. 
\end{proof}
\begin{Lema}\label{simplemausimporta}
Let $Q'\in\Psi_n$ such that $\epsilon(Q)<\epsilon(Q')$. For any $f\in K[x]$ let
\[
f=a_0+a_1Q+\ldots+a_rQ^r\mbox{ and }f=b_0+b_1Q'+\ldots+b_rQ'^r
\]
be the $Q$ and $Q'$-expansions of $f$, respectively. For $l= \delta_Q(f)$ we have
\[
\nu(a_l-b_l)>\nu(a_l).
\]
In particular, $\nu(a_l)=\nu(b_l)$.
\end{Lema}
\begin{proof}
Let $h:=Q-Q'$ so that $Q=Q'+h$. Then
\[
a_iQ^i=\sum_{j=0}^i {i\choose j}a_ih^jQ'^{i-j}.
\]
For each $i$, $0\leq i\leq r$, and $j$, $0\leq j\leq i$, let
\begin{equation}\label{equansmuitobunitinha}
 {i\choose j}a_ih^j=a_{ij0}+a_{ij1}Q'+\ldots+a_{ijn}Q'^n
\end{equation}
be the $Q'$ expansion of $a_ih^j$. Then
\[
b_l=\sum_{i-j+k=l}a_{ijk}.
\]
For $i$, $0\leq i\leq r$, and $j$, $0\leq j\leq i$, if $k:=l+j-i>0$, then by \eqref{eqabespivmeseugost}, we have
\begin{displaymath}
\begin{array}{rcl}
\nu(a_{ijk})+k\nu(Q')&>&\nu(a_ih^j)+d(B-\nu(Q))\\[8pt]
&\geq& \nu(a_iQ^i)+(j-i)\nu(Q)+k(B-\nu(Q))\\[8pt]
&\geq& \nu(a_lQ^l)+(j-i)\nu(Q)+k(B-\nu(Q))\\[8pt]
&\geq& \nu(a_l)+(l+j-i)\nu(Q)+k(B-\nu(Q))=\nu(a_l)+kB.		
\end{array}
\end{displaymath}
Since $B>\nu(Q')$ we have $\nu(a_{ijk})>\nu(a_l)$.

Suppose now that $k:=l+j-i=0$ (i.e., that $i=l+j$). If $j=0$, then by definition $a_{ijk}=a_l$. If $j>0$, then $i>l$. Since $l=\delta_Q(f)$ we have $\nu(a_iQ^i)>\nu(a_lQ^l)$. Then by Lemma \ref{sugestnart}, applied to \eqref{equansmuitobunitinha}, we have
\[
\nu_Q(a_{ijk}Q'^k)\geq \nu_Q(a_ih^j)=\nu(a_iQ^i)+(j-i)\nu(Q)>\nu(a_l)+(l+j-i)\nu(Q).
\]
Since $\nu_Q(Q')=\nu(Q)$ we obtain that $\nu(a_{ijk})>\nu(a_l)$ and the result follows.
\end{proof}

For each $\rho<\lambda$ and $f\in K[x]$, let
\[
f=a_{\rho 0}(f)+a_{\rho 1}(f)Q_\rho+\ldots+a_{\rho r}(f)Q_\rho^r
\]
be the $Q_\rho$-expansion of $f$. The value of $a_{\rho 0}(f)$ will be very important in what follows.
\begin{Prop}\label{spivdevemostrarpamim}
For $f\in K[x]$ with $\deg_X(f)\leq d$, write $f=l(Q)$ for $l(X)\in K[x]_n[X]$. For every $\rho<\lambda$ we have
\begin{equation}\label{eqnatudofucn}
\nu_\rho(l(h_\rho))\geq\nu_\rho(f).
\end{equation}
Moreover, the equality holds in \eqref{eqnatudofucn} if and only if
\[
\nu\left(a_{\rho 0}(f)\right)=\nu_\rho(f)=\nu(l(h_\rho)).
\]
\end{Prop}
\begin{proof}
By definition
\[
f=l(Q)=l(Q_\rho+h_\rho)=Q_\rho p(x)+l(h_\rho)\mbox{ for some }p(x)\in K[x].
\]
Hence $a_{\rho 0}(f)=a_{\rho 0}(l(h_\rho))$. Let
\begin{equation}\label{eqaimport1}
l(h_\rho)=a_{\rho 0}(f)+b_1Q_\rho+\ldots+b_lQ_\rho^l 
\end{equation}
be the $Q_\rho$-expansion of $l(h_\rho)$. We will show that $\nu(b_iQ_\rho^i)>\nu_\rho(f)$ for every $i$, $1\leq i\leq l$, and this will imply our result.

Let $f=a_0+a_1Q+\ldots+a_rQ^r$ be the $Q$-expansion of $f$, so that
\begin{equation}\label{eqaimport2}
l(h_\rho)=a_0+a_1h_\rho+\ldots+a_rh_\rho^r.
\end{equation}
For each $j$, $1\leq j\leq r$, consider the $Q_\rho$-expansion
\begin{equation}\label{eqaimport3}
a_jh_\rho^j=a_{\rho 0j}+a_{\rho 1j}Q_\rho+\ldots+a_{\rho lj}Q_\rho^l
\end{equation}
of $a_jh_\rho^j$. Comparing \eqref{eqaimport1}, \eqref{eqaimport2} and \eqref{eqaimport3}, it is enough to show that
\[
\nu\left(a_{\rho i j}Q_\rho^i\right)>\nu\left(a_{\rho0}(f)\right)\mbox{ for every }i,j, 1\leq i\leq l\mbox{ and } 1\leq j\leq r.
\]
For a fixed $j$, $1\leq j\leq r$, by \eqref{eqabespivmeseugost} applied to \eqref{eqaimport3} we have
\begin{equation}\label{primisriaapeimpo}
\nu\left(a_{\rho ij}Q_\rho^j\right)>\nu\left(a_jQ^j\right)+s(B-\nu(Q)).
\end{equation}
Since $\nu(Q)=\nu(h_\rho)=\nu_\rho(Q)$, if $\nu(a_0)<\nu(a_iQ^i)$ for every $i$, $1\leq i\leq r$, then
\[
\nu_\rho(f)=\nu(a_0)=\nu(l(h_\rho))
\]
and we are done. Suppose not and take $l=\delta_Q(f)>0$. By \eqref{primisriaapeimpo} and the fact that $\nu(a_{\rho l}(f))=\nu(a_l)$ (Lemma \ref{simplemausimporta}), we have
\begin{displaymath}
\begin{array}{rcl}
\nu\left(a_{\rho ij}Q_\rho^j\right)&>& \nu(a_jQ^j)+l(B-\nu(Q))\geq\nu(a_l)+l \nu(Q)+l(B-\nu(Q))\\[8pt]
&=&\nu(a_l)+l B>\nu(a_l)+ l\nu(Q_\rho)\geq\nu_{\rho}(f).
\end{array}
\end{displaymath}
This completes the proof.
\end{proof}

\begin{Cor}\label{Corolarcurucial}
If $\deg(f)< \deg(F)$, then there exists $\rho$ such that
\[
\nu(l(h_\sigma))=\nu(f)=\nu_\sigma(f)=\nu(a_{\sigma 0}(f))
\]
for every $\sigma$, $\rho< \sigma<\lambda$.
\end{Cor}
\begin{proof}
It is well-known that if $f$ is $\Psi_n$-stable, then there exists $Q'\in \Psi_n$ such that $0\in S_{Q'}(f)$. The result follows immediately.
\end{proof}
\section{The Taylor expansion of a polynomial}
We will consider the ring $K(x)[X]$ where $X$ is an indeterminate and let $\partial_i$ denote the $i$-th  Hasse derivative with respect to $X$. Then, for every $l(X)\in K(x)[X]$ and $a,b\in K[x]$ we have the Taylor expansion
\[
l(b)=l(a)+\sum_{i=1}^{\deg_Xl}\partial_il(a)(b-a)^i.
\]

\begin{Lema}[Lemma 4 of \cite{Kap}]\label{lemasuperutil}
Let $\Gamma$ be an ordered abelian group, $\beta_1,\ldots,\beta_n\in \Gamma$ and $\{\gamma_\rho\}_{\rho<\lambda}$ an increasing sequence in $\Gamma$, without a last element. If $t_1,\ldots,t_n$ are distinct positive integers, then there exist $b$, $1\leq b\leq n$, and $\rho<\lambda$, such that
\[
\beta_i+t_i\gamma_\sigma>\beta_b+t_b\gamma_\sigma\mbox{ for every }i, 1\leq i\leq n, i\neq b\mbox{ and }\sigma>\rho.
\]
\end{Lema}

\begin{Prop}\label{Corsobrepowerofp}
Take $f\in K[x]$ such that $\deg_X(f)\leq d$ and write $f=l(Q)$ for $l(X)\in K[x]_n[X]$. Then there exists $\rho<\lambda$ such that
\[
\nu_\sigma(f)=\nu(l(h_\sigma))\mbox{ for every }\sigma,\rho<\sigma<\lambda.
\]
In particular, $f$ is $\Psi_n$-stable if and only if it is $(\Psi_n,Q)$-fixed.
\end{Prop}
\begin{proof}
Since for every $j$, $1\leq j\leq \deg_X(l)$, we have $\deg(\partial_jl(Q))<\deg(F)$ we can use Corollary \ref{Corolarcurucial} to obtain $\rho<\lambda$ such that
\begin{equation}\label{fixandodderi}
\beta_j:=\nu_\sigma\left(\partial_jl(h_\sigma)\right)=\nu\left(\partial_jl(h_\sigma)\right)\mbox{ for every }j,1\leq j\leq \deg_X(l)\mbox{ and }\rho<\sigma.
\end{equation}
By Lemma \ref{lemasuperutil}, there exist $b$, $1\leq b\leq \deg_X(l)$, and $\rho<\lambda$ such that for every $\lambda$, $\rho<\sigma<\lambda$ and $i\neq b$ we have
\begin{equation}\label{eanfgepalsiemba2}
\beta_b+b\gamma_\rho<\beta_i+i\gamma_\rho\mbox{ and }\eqref{fixandodderi}\mbox{ is satisfied}.
\end{equation}
For $\sigma$, $\rho<\sigma<\lambda$, since
\[
f-l(h_\sigma)=\sum_{i=1}^{\deg_X(f)}\partial_il(h_\sigma)Q_\sigma^i
\]
we have
\begin{equation}\label{espanquuspa1}
\nu_\sigma\left(f-l(h_\sigma)\right)=\beta_b+b\gamma_\sigma
\end{equation}
and
\begin{equation}\label{espanquuspa2}
\nu\left(f-l(h_\sigma)\right)=\beta_b+b\gamma_\sigma.
\end{equation}
By Proposition \ref{spivdevemostrarpamim} we have $\nu_\sigma(l(h_\sigma))\geq \nu_\sigma(f)$. This and \eqref{espanquuspa1} imply that
\begin{equation}\label{maisumaesqufoido}
\nu(l(h_\sigma))\geq \nu_\sigma(l(h_\sigma))\geq \beta_b+b\gamma_\sigma.
\end{equation}

If $\nu_\sigma(f)=\nu(f)$, then $\nu(l(h_\sigma))\geq \nu(f)$. Suppose, aiming for a contradiction, that $\nu(l(h_\sigma))> \nu(f)$. Then by \eqref{espanquuspa2} we have
\[
\nu(f)=\beta_b+b\gamma_\sigma.
\]
For any $\sigma'>\sigma$, by \eqref{maisumaesqufoido} we would obtain
\[
\nu(l(h_{\sigma'}))\geq \beta_b+b\gamma_{\sigma'}>\beta_b+b\gamma_{\sigma}=\nu(f)
\]
and this contradicts \eqref{espanquuspa2} (with $\sigma$ replaced by $\sigma'$). Hence, $\nu(l(h_\sigma))=\nu(f)$. 

Suppose now that $\nu_\sigma(f)<\nu(f)$. If $\nu_\sigma(l(h_\sigma))> \nu_\sigma(f)$, then by \eqref{espanquuspa1} we have $\nu_\sigma(f)=\beta_b+b\gamma_\sigma$. Consequently,
\[
\nu(f-l(h_\sigma))\geq\min\{\nu(f),\nu(l(h_\sigma))\}>\nu_\sigma(f)=\beta_b+b\gamma_\sigma,
\]
what is a contradiction to \eqref{espanquuspa2}. Hence,
\[
\nu_\sigma(l(h_\sigma))=\nu_\sigma(f).
\]
By the second part of Proposition \ref{spivdevemostrarpamim} we obtain that $\nu_\sigma(f)=\nu(l(h_\sigma))$. This, \eqref{espanquuspa1}--\eqref{maisumaesqufoido} and the fact that $\nu_\sigma(f)<\nu(f)$ imply that
\[
\nu_\sigma(f)=\nu(l(h_\sigma))=\beta_b+b\gamma_\sigma.
\]
\end{proof}
\begin{Obs}\label{obsequsajudabasta}
In the proof of Theorem \ref{Tehoremquasneunpas} we will use the explicit calculation of $\nu(l(h_\sigma))$ obtained in the previous proposition.
\end{Obs}

\section{Proof of Theorem \ref{Tehoremquasneunpas}}
We will adapt the proof by Kaplansky in \cite{Kap}. For each $i$, $1\leq i\leq d$, the polynomial $\partial_iL(Q)$ has degree smaller than $\deg(F)$, hence by Corollary \ref{Corolarcurucial} there exists $\rho_0<\lambda$ such that
\begin{equation}\label{equanajudamuito}
\beta_i:=\nu(\partial_iL(Q))=\nu(\partial_iL(h_{\rho}))
\end{equation}
for every $\rho$, $\rho_0<\rho<\lambda$.
\begin{Lema}\label{Lemmasobrepowerofp}
If $i=p^t$ and $j=p^tr$ with $r>1$ and $p\nmid r$, then there exists $\rho<\lambda$ such that
\[
\beta_i+i\gamma_\sigma< \beta_j+j\gamma_\sigma\mbox{ for every }\sigma, \rho<\sigma<\lambda.
\]
Moreover, if $C$ in the value group of $\nu$ is such that $C>\gamma_\rho$ for every $\rho<\lambda$, then 
\[
\beta_i+iC<\beta_j+jC.
\]
\end{Lema}
\begin{proof}
From the Taylor formula (applied to $\partial_iL)$ we have
\begin{displaymath}
\begin{array}{rcl}
\partial_iL(h_\sigma)-\partial_iL(h_\rho)&=& \displaystyle\sum_{k=1}^{n-i}\partial_k\partial_iL(h_\rho)(h_{\sigma}-h_{\rho})^k\\[10pt]
&=& \displaystyle\sum_{k=1}^{n-i}{i+k\choose i}\partial_{i+k}L(h_\rho)(h_{\sigma}-h_{\rho})^k.\\[8pt]
\end{array}
\end{displaymath}
By Lemma \ref{lemasuperutil}, for $\rho<\sigma$ large enough
\begin{equation}
\nu\left(\partial_iL(h_\sigma)-\partial_iL(h_\rho)\right)=\min_{1\leq k\leq n-i}\left\{\nu\left({i+k\choose i}\partial_{i+k}L(h_\rho)(h_\sigma-h_\rho)^k\right)\right\}.
\end{equation}
In particular, taking $k=j-i$, this gives
\begin{equation}\label{eqcombinattaylo}
\nu\left(\partial_iL(h_\sigma)-\partial_iL(h_\rho)\right)\leq \nu\left({j\choose i}\partial_{j}L(h_\rho)(h_\sigma-h_\rho)^{j-i}\right)
\end{equation}
By \eqref{equanajudamuito} and \eqref{eqcombinattaylo} we have
\begin{displaymath}
\begin{array}{rcl}
\beta_i &\leq&\nu(\partial_iL(h_\sigma)-\partial_iL(h_\rho))\\[8pt]
&\leq&\displaystyle\nu\left({j\choose i}\partial_jL(h_\rho)(h_\sigma-h_\rho)^{j-i}\right)\\[8pt]
&= &\displaystyle\nu\left({j\choose i}\right)+\beta_j+(j-i)\gamma_\rho.
\end{array}
\end{displaymath}
Since $\displaystyle p\nmid {j\choose i}$ and $\CA(K\nu)=p$ we have $\displaystyle\nu\left({j\choose i}\right)=0$. Consequently,
\[
\beta_i\leq \beta_j+(j-i)\gamma_\rho.
\]
This means that for every $\sigma$, $\rho<\sigma<\lambda$, we have
\[
\beta_i+i\gamma_\sigma<\beta_j+j\gamma_\sigma.
\]
Take $C>\gamma_\rho$ for every $\rho<\lambda$. If $\beta_i+iC\geq\beta_j+jC$, then
\[
\beta_i-\beta_j\geq (j-i)C>(j-i)\gamma_\sigma\mbox{ for every }\rho<\lambda
\]
and this contradicts the first part.
\end{proof}

The proof of the next result is very similar to the proof of Proposition \ref{spivdevemostrarpamim}.
\begin{Lema}\label{Lemaparacabcomtuo}
Fix $\theta<\lambda$ and for each $i$, $0\leq i\leq r$, set $a_{i0}:=a_{\theta 0}(\partial_iL(h_\theta))$. Then
\[
\nu_\theta(\partial_iL(h_\theta)-a_{i0})+i\nu(Q_\theta)>\overline B.
\]
\end{Lema}
\begin{proof}
Since $F=a_0+a_1Q+\ldots+a_rQ^r$ we have
\begin{displaymath}
\begin{array}{rcl}
\partial_0L(h_\theta)&=&a_0+a_1h_\theta+\ldots+a_rh_\theta^r\\[8pt]
\partial_1L(h_\theta)&=&a_1+\ldots+ra_rh_\theta^{r-1}\\
&\vdots&\\
\partial_rL(h_\theta)&=&a_r.
\end{array}
\end{displaymath}
If we write
\[
\partial_iL(h_\theta)=b_{i0}+b_{i1}h_\theta+\ldots+b_{is}h_\theta^s,
\]
then
\[
\nu\left(b_{ij}h_\theta^j\right)+i\nu(Q)\geq \nu(a_{i+j}Q^{i+j})\geq \nu_Q(F).
\]
For each $j>1$, write
\[
b_{ij}h_\theta^j=a_{i0j}+a_{i1j}Q_\theta+\ldots+a_{isj}Q_\theta^s.
\]
For every $i,j$ and $k>0$, by \eqref{eqabespivmeseugost2} we have
\begin{displaymath}
\begin{array}{rcl}
\nu\left(a_{ikj}Q^k_\theta\right)+i\nu(Q_\theta)&>&\nu(b_{ij}h_\theta^j)+\overline B-\nu_Q(F)+i\nu(Q_\theta)\\[8pt]
&>&\nu_Q(F)+\overline B-\nu_Q(F)= \overline B.
\end{array}
\end{displaymath}
For every $i$, we have
\[
\partial_iL(h_\theta)-a_{i0}=\sum_{k,j>0}a_{ikj}Q_\theta^j.
\]
The result follows.
\end{proof}
We proceed now with the proof of Theorem \ref{Tehoremquasneunpas}.
\begin{proof}[Proof of Theorem \ref{Tehoremquasneunpas}]
For each $i=p^s$ with $1\leq i\leq \deg(F)$ and $j=p^sr$, with $p\nmid r$, by Lemma \ref{Lemmasobrepowerofp} we have
\begin{equation}\label{eqsobreelemenselim}
\beta_i+iB<\beta_j+jB\mbox{ for }\rho\mbox{ large enough}.
\end{equation}
Then there exists $\rho_{ij}$ such that for every $\sigma>\rho_{ij}$ we have
\begin{equation}\label{equajudaprafrente}
\beta_i+i\gamma_\rho<\beta_j+j\gamma_\sigma\mbox{ for every }\rho<\lambda.
\end{equation}
Take $\sigma$ such that $\eqref{equajudaprafrente}$ is satisfied for every $i=p^s$ and $j=p^sr$, with $p\nmid r$. Write
\[
I=\{l\mid 1\leq l\leq d\mbox{ such that }l=p^i\mbox{ for some }i\in\N\}
\]
and
\[
J:=\{1,\ldots,d\}\setminus I.
\]
Then, for every $j\in J$ there exists $i\in I$ such that
\[
\beta_i+i\gamma_\rho<\beta_j+j\gamma_\sigma\mbox{ for every }\rho<\lambda.
\]
This means that for every $\rho>\sigma$ we have
\begin{equation}\label{eqquennaoentsbanada}
\nu\left(\sum_{j\in J}\partial_jL(h_\sigma)(h_\rho-h_\sigma)^j\right)\geq\min_{j\in J}\{\beta_j+j\gamma_\sigma\}>\min_{i\in I}\{\beta_i+i\gamma_\rho\}=\beta_b+i\gamma_b.
\end{equation}

In order to prove \textbf{(i)}, take $\theta>\sigma$ and consider the polynomial
\[
F_p(x)=L(h_\theta)+\sum_{i\in I}\partial_{i}L(h_\theta)(Q-h_\theta)^{i}=:L_p(Q).
\]
Then
\[
L(h_\rho)-L_p(h_\rho)=\sum_{j\in J}\partial_{j}L(h_\theta)(h_\rho-h_\theta)^{j}.
\]
Consequently,
\[
\nu\left(L(h_\rho)-L_p(h_\rho)\right)>\beta_b+b\gamma_\rho=\nu(L(h_\rho)).
\]
The last equality follows from the proof of Proposition \ref{Corsobrepowerofp} (as observed in Remark \ref{obsequsajudabasta}). Hence,
\[
\nu(L_p(h_\rho))=\beta_b+b\gamma_\rho.
\]
If $r\notin I$, then $\deg_X(F_p)<d$ so we apply Lemma \ref{Corsobrepowerofp} to obtain that $F_p\in S_n$. This is a contradiction to the minimality of the degree of $F$ in $S_n$. Hence, $\deg_X(F_p)= d$ and $F_p$ is monic. Since $\deg_X(F_p)\leq d$, by Lemma \ref{Corsobrepowerofp} we obtain that $F_p\in S_n$ and so is a limit key polynomial for $\Psi_n$.

In order to prove \textbf{(ii)}, for each $i\geq 0$ take
\[
a_i=a_{\theta 0}\left(\partial_iL(h_\theta)\right)
\]
and
\[
\overline F_p:=\sum_{i\in I\cup\{ 0\}}a_iQ_\theta^i.
\]
By Lemma \ref{Lemaparacabcomtuo} we have
\[
\nu_\rho\left(F_p-\overline{F}_p\right)> \overline B>\nu_{\rho}\left(F\right)\mbox{ for every }\rho, \theta<\rho<\lambda.
\]
As before, we conclude that $\overline{F}_p$ is a limit key polynomial for $\Psi_n$ and this completes the proof.
\end{proof}
\begin{Obs}
One can prove (Proposition 3.5 of \cite{michael}) that $a_r=1$ and in particular $\deg(F)=n\deg_X(F)$. 
\end{Obs}


\begin{thebibliography}{99}

\bibitem{Matheus} M. dos S. Barnab\'e and J. Novacoski, \textit{Generating sequences and key polynomials}, to appear in Michigan Mathematical Journal, arXiv:2007.12293, 2020.

\bibitem{mabjulispiv} J. Decaup, W. Mahboud and M. Spivakovsky, \textit{Abstract key polynomials and comparison theorems with the key polynomials of Mac Lane-Vaquié}, Illinois Journal of Mathematics \textbf{62} (2018), 253--270.

\bibitem{MaSpiv} F.J. Herrera Govantes, W. Mahboub, M.A. Olalla Acosta and M. Spivakovsky, \textit{Key polynomials for simple extensions of valued fields}, arXiv:1406.0657, 2014.

\bibitem{Kap} I. Kaplansky, \textit{Maximal fields with valuations I}, Duke Math. Journ. \textbf{9}
(1942), 303 -- 321.

\bibitem{michael} M. Moraes and J. Novacoski, \textit{Limit key polynomials as $p$-polynomials}, J. Algebra \textbf{579}, 152--173 (2021).

\bibitem{nart} E. Nart, \textit{Key polynomials over valued fields}, Publ. Mat. 64 (2020), 195--232.

\end{thebibliography}
\end{document}